\documentclass[12pt]{article}

\usepackage{amssymb,tikz,float,enumerate}
\usepackage{amsthm,amsfonts, amsbsy, amssymb,amsmath,graphicx}
\usepackage{graphics}
\usepackage[english]{babel}
\usepackage{graphicx}
\usepackage{enumerate}
\usepackage{tikz}
\usepackage{hyperref}
\usepackage{color}
\usepackage[T1]{fontenc}
\usepackage[margin=0.75in]{geometry}

\newtheorem{theorem}{\bf Theorem}[section]

\newtheorem{corollary}[theorem]{\bf Corollary}

\title{A note on the packing chromatic number\\ of lexicographic products\thanks{The second author was partially supported by Slovenian research agency under the grants P1-0297 and J1-9109.}}

\author{Dragana Bo\v{z}ovi\'c$^{(1)}$ and Iztok Peterin$^{(1,2)}$\\
\\
$^{(1)}$ {\small Faculty of Electrical Engineering and Computer Science}\\
{\small University of Maribor,} {\small Koro\v{s}ka cesta 46, 2000 Maribor, Slovenia.} \\
$^{(2)}${\small Institute of Mathematics, Physics and Mechanics}\\
{\small Jadranska ulica 19, 1000 Ljubljana, Slovenia.} \\
{\small\texttt{e-mails:}\small\it dragana.bozovic\@@um.si and iztok.peterin\@@um.si} \\
}

\date{}

\begin{document}

\maketitle

\begin{abstract}
The packing chromatic number $\chi_{\rho}(G)$ of a graph $G$ is the smallest integer $k$ such
that there exists a $k$-vertex coloring of $G$ in which any two vertices receiving color
$i$ are at distance at least $i+1$. In this short note we present upper and lower bound for the packing chromatic number of the lexicographic product $G\circ H$ of graphs $G$ and $H$. Both bounds coincide in many cases. In particular this happens if $|V(H)|-\alpha(H)\geq {\rm diam}(G)-1$, where $\alpha(G)$ denotes the independence number of $G$.
\end{abstract}



\section{Introduction and preliminaries} 
Let $G$ be a simple graph. To shorten the notation we use $|G|$ instead of $|V(G)|$ for the order of $G$. The \emph{distance} $d_G(u,v)$ between vertices $u$ and $v$ of $G$ is the length of a shortest path between $u$ and $v$ in $G$. The \emph{diameter} of $G$ is denoted by ${\rm diam}(G)$ and is the maximum length of a shortest path between any two vertices of $G$.

Let $t$ be a positive integer. A set $X\subseteq V(G)$ is a $t$-\emph{packing} if any two different vertices from $X$ are at distance more than $t$. The $t$-\emph{packing number} of $G$, denoted by $\rho_t(G)$, is the maximum cardinality of a $t$-packing of $G$. Notice, that if $t=1$, then the $1$-packing number equals to the independence number $\alpha (G)$ and we use the later more common notation for it. An independent set of cardinality $\alpha(G)$ is called $\alpha(G)$-\emph{set}. 
The \emph{packing chromatic number} $\chi_{\rho}(G)$ of $G$ is the smallest integer $k$ such that $V(G)$ can be partitioned into subsets $X_1,\dots, X_k$, where $X_i$ induces an $i$-packing for every $1\leq i\leq k$. Another approach is from a $k$-\emph{packing coloring} of $G$, which is a function $c: V(G) \rightarrow [k]$, where $[k]=\{1,\dots, k\}$, such that if $c(u)=c(v)=i$, then $d_G(u,v)>i$. Clearly, $\chi_{\rho}(G)$ is the minimum integer $k$ for which a $k$-packing coloring of $G$ exists. 

The concept of packing chromatic number was introduced by Goddard et al. in \cite{GHHHR} under the name broadcast chromatic number. The problem of determining the packing chromatic number of a graph is a very difficult problem and is NP-complete even for trees as shown in \cite{FiGo}. The attention was fast drawn to Cartesian product and infinite latices like hexagonal, triangular and similar. In \cite{bresar-klavzar-rall-1} it was shown that the packing chromatic number of an infinite hexagonal lattice lies between 6 and 8. Upper bound was later improved to 7 in \cite{fiala-klavzar-lidicky} and finally settled to 7 in \cite{korze-vesel}. For infinite triangular lattice and three-dimensional integer lattice $\mathbb{Z}^3$ the packing chromatic number is infinite as shown in \cite{finbow-rall}. The packing chromatic number of the Cartesian product was already considered in \cite{bresar-klavzar-rall-1} where the general upper and lower bound were set. The lower bound was later improved in \cite{jacobs-jonck-joubert}. Several exact values and bounds for special families of Cartesian product graphs can be found in \cite{jacobs-jonck-joubert,korze-vesel}. 

In this note we switch from Cartesian to lexicographic product and prove an upper and a lower bound for the packing chromatic number of lexicographic product. It turns out that these two bounds coincide in many cases. In particular, if ${\rm diam}(G)\leq 2$, if ${\rm diam}(G)=3$ and $H\ncong \overline{K}_n$ and if $|V(H)|-\alpha(H)\geq {\rm diam}(G)-1$ and $H\ncong \overline{K}_n$. 

The \emph{lexicographic product} of graphs $G$ and $H$ is the graph $G\circ H$ (also sometimes denoted with $G[H]$) with the vertex set $V(G) \times V(H)$. Two vertices $(g_1,h_1)$ and $(g_2,h_2)$ are adjacent if either $g_1g_2 \in E(G)$ or $g_1 = g_2$ and $h_1h_2 \in E(H)$. Set $G^h=\{(g,h):g\in V(G)\}$ is called a $G$-\emph{layer through} $h$ and $H^g=\{(g,h):h\in V(H)\}$ is called an $H$-\emph{layer through} $h$. Clearly, subgraphs of $G\circ H$ induced by $G^h$ and $H^g$ are isomorphic to $G$ and $H$, respectively. The distance between two vertices in lexicographic product is given by
\begin{equation}\label{dist}
d_G((g_1,h_1),(g_2,h_2))=
\left\{ 
\begin{array}{ccc}
d_G(g_1,g_2)& : & g_1 \neq g_2 \\  
\min\{2,d_H(h_1,h_2)\}& : & g_1=g_2\\
\end{array}%
\right. 
\end{equation}
and depends heavily on the distance between projections of both vertices to $G$. For the independence number it is well known that 
\begin{equation}\label{indep}
\alpha(G \circ H)=\alpha(G)\alpha(H),
\end{equation}
see Theorem 1 in \cite{geller-stahl}.
Lexicographic product $G\circ H$ is connected if and only if $G$ is connected. For more properties of the lexicographic product see~\cite{hammack-imrich-klavzar}.

\section{Results} 

In this section we present a lower and an upper bound for the packing chromatic number of lexicographic product of graphs. We start with the lower bound and we use the following notation
\begin{equation*}
d(G)=
\left\{ 
\begin{array}{ccc}
1& : & G\cong K_n \\  
{\rm diam}(G)-1& : & \textnormal{otherwise}\\
\end{array}%
\right. .
\end{equation*}

\begin{theorem} \label{spodnjameja}
If $G$ and $H$ are graphs, then
$$\chi_{\rho}(G\circ H)\geq |G|\cdot|H|-\alpha(G)\alpha(H)-\sum_{i=2}^{{\rm diam}(G)-1} \rho_i(G)+d(G).$$
\end{theorem}

\begin{proof}
Denote $\ell=|G|\cdot|H|-\alpha(G)\alpha(H)-\sum_{i=2}^{{\rm diam}(G)-1} \rho_i(G)+d(G)$. Let $X_1,\dots, X_k$ be a partition of $V(G\circ H)$ that yields a $k$-packing coloring of $G\circ H$. 
We have at most $\alpha(G)\alpha(H)$ vertices in $X_1$ by (\ref{indep}). Denote by $R_i$ a $\rho_i(G\circ H)$-set for $2\leq i \leq {\rm diam}(G)-1$. By (\ref{dist}) we have $|H^g\cap R_i|\leq 1$ for every $g\in V(G)$. So there are at most $\rho_i(G)$ vertices in $X_i$ for $2\leq i \leq {\rm diam}(G)-1$.
For $i\geq {\rm diam}(G)$ there can only be one vertex in $X_i$ since all the vertices are at distance at most ${\rm diam}(G)$ from vertex in $X_i$. So we have at most $\alpha(G)\alpha(H)$ vertices colored with color 1, at most $\rho_i(G)$ vertices colored with color $i$ for every $2\leq i\leq{\rm diam}(G)-1$, and we need one color for each one of the remaining vertices and there are $|G|\cdot|H|-\alpha(G)\alpha(H)-\sum_{i=2}^{{\rm diam}(G)-1} \rho_i(G)$ of them. Meaning that $\chi_{\rho}(G\circ H)\geq \ell$ because we have exactly $d(G)$ color classes which possibly have more than one vertex. 
\end{proof}

We continue with an upper bound that has a similar structure as the lower bound from Theorem \ref{spodnjameja}.

\begin{theorem} \label{zgornjameja}
Let $G$ and $H$ be graphs and $k=|H|-\alpha(H)$. If $H\ncong\overline{K}_n$, then
$$\chi_{\rho}(G\circ H)\leq |G|\cdot|H|-\alpha(G)\alpha(H)-\sum_{i=2}^{k+1} \rho_i(G)+k+1.$$
\end{theorem}

\begin{proof}
Denote $\ell=|G|\cdot|H|-\alpha(G)\alpha(H)-\sum_{i=2}^{k+1} \rho_i(G)+k+1$. 
We know that $\alpha(G \circ H)=\alpha(G)\alpha(H)$ and it is easy to see that $\alpha(G\circ H)$-set can be written as $A_G\times A_H$ where $A_G$ is an $\alpha(G)$-set and $A_H$ is an $\alpha(H)$-set. We color all the vertices from $A_G \times A_H$ with color 1.
Let $k=|H|-\alpha(H)$. There remain $k$ $G$-layers with no colored vertices. In each of those layers we color $\rho_i(G)$ vertices with color $i$, $2\leq i \leq k+1$ (one color $i$ is used in one layer). Each of the remaining uncolored vertices is colored with its own color.  
So we have $\alpha(G)\alpha(H)$ vertices colored with color 1, $\rho_i(G)$ vertices colored with color $i$ for every $2\leq i \leq k+1$, and we need one color for each one of the remaining uncolored vertices. Clearly, there are $|G|\cdot|H|-\alpha(G)\alpha(H)-\sum_{i=2}^{k+1} \rho_i(G)$ vertices colored with its own color. Meaning that $\chi_{\rho}(G\circ H)\leq \ell$ because we have $k+1$ color classes which possibly have more than one vertex.
\end{proof}

Notice that if ${\rm diam}(G)\leq 2$, then also ${\rm diam}(G\circ H)\leq 2$ (see \ref{dist}) and only color 1 can appear more then once in any packing coloring. Therefore, if ${\rm diam}(G)\leq 2$,  Theorem \ref{zgornjameja} also holds for $H\cong \overline{K}_n$. 
Next we show that if the number of vertices of $H$ without its $\alpha(H)$-set is comparable with ${\rm diam}(G)$, then both bounds coincide. 

\begin{corollary} \label{posledicaizreka1}
Let $G$ and $H$ be graphs and $|H|-\alpha(H)\geq {\rm diam}(G)-1$. If $H\ncong\overline{K}_n$, then
$$\chi_{\rho}(G\circ H)=|G|\cdot|H|-\alpha(G)\alpha(H)- \sum_{i=2}^{{\rm diam}(G)-1} \rho_i(G) +{\rm diam}(G)-1.$$ 
\end{corollary}

\begin{proof}
Let first $G\cong K_n$. By Theorem~\ref{spodnjameja} it holds that $\chi_{\rho}(G\circ H)\geq n|H|-\alpha(H)-\sum_{i=2}^{0} \rho_i(G)+1=n|H|-\alpha(H)+1$ since $d(G)=1$. On the other hand let $k=|H|-\alpha(H)$ and we have $\chi_{\rho}(G\circ H)\leq n|H|-\alpha(H)-(k+1-2+1)+k+1=n|H|-\alpha(H)+1$ by Theorem~\ref{zgornjameja} since $\rho_i(G)=1$ for every $2\leq i \leq k+1$. Hence, the equality follows when $G\cong K_n$.

\noindent Otherwise $G\not \cong K_n$ and $d(G)={\rm diam}(G)-1$. So by Theorem~\ref{spodnjameja} it holds that $\chi_{\rho}(G\circ H)\geq |G|\cdot|H|-\alpha(G)\alpha(H)-\sum_{i=2}^{{\rm diam}(G)-1} \rho_i(G)+{\rm diam}(G)-1$. Since $k\geq {\rm diam}(G)-1$ and $\rho_i(G)=1$ for every ${\rm diam}(G) \leq i \leq k+1$, by Theorem~\ref{zgornjameja} it holds that 
\begin{equation}
\begin{split}
\chi_{\rho}(G\circ H) &\leq |G|\cdot|H|-\alpha(G)\alpha(H)- \left ( \sum_{i=2}^{{\rm diam}(G)-1} \rho_i(G) + \sum_{i={\rm diam}(G)}^{k+1} \rho_i(G)\right ) +k+1 = \\
&=|G|\cdot|H|-\alpha(G)\alpha(H)- \sum_{i=2}^{{\rm diam}(G)-1} \rho_i(G) - (k+1-{\rm diam}(G)+1) +k+1 = \\
&=|G|\cdot|H|-\alpha(G)\alpha(H)- \sum_{i=2}^{{\rm diam}(G)-1} \rho_i(G) +{\rm diam}(G)-1.
\end{split}
\end{equation}
\end{proof}

We can expect that the condition of Corollary~\ref{posledicaizreka1} will be fulfilled more frequently when ${\rm diam}(G)$ is small. In particular, for ${\rm diam}(G)=1$ the condition is always satisfied and we have
$$\chi_{\rho}(G\circ H)=|G|\cdot|H|-\alpha(H)+1$$ as seen in the proof of the previous corollary.
Notice that in the case of ${\rm diam}(G)=2$ the sum in the lower bound of Theorem \ref{spodnjameja} does not exist and that $d(G)=1$. Also $\rho_i(G)=1$ for every $2\leq i\leq k$ since ${\rm diam}(G)=2$. Therefore we have $-\sum_{i=2}^{k+1} \rho_i(G)+k+1=-(k+1-2+1)+k+1=1$ in the upper bound of Theorem \ref{zgornjameja}. Hence both bounds coincide and we have the following corollary.   

\begin{corollary} \label{posledicaizreka2}
Let $G$ and $H$ be graphs. If ${\rm diam}(G)=2$, then 
$$\chi_{\rho}(G\circ H)=|G|\cdot|H|-\alpha(G)\alpha(H)+1.$$
\end{corollary}


Similar holds also when ${\rm diam}(G)=3$. Namely in this case ${\rm diam}(G\circ H)=3$ by (\ref{dist}) and only two color classes ($X_1$ and $X_2$) can have more than one representative. Therefore bounds from Theorems \ref{zgornjameja} and \ref{spodnjameja} coincide again under condition that there is at least one $G$-layer without vertices from $X_1$. This always occurs if $H\ncong\overline{K}_n$ and the following corollary holds.

\begin{corollary} \label{posl}
Let $G$ and $H$ be graphs. If ${\rm diam}(G)=3$ and $H\ncong\overline{K}_n$, then 
$$\chi_{\rho}(G\circ H)=|G|\cdot|H|-\alpha(G)\alpha(H)-\rho_2(G)+2.$$
\end{corollary}

Continuing in this manner things get more complicated. Therefore we finish with an approach from the different side and concentrate on a family of graphs with big diameter, namely the case when $G\cong P_n$. For this we first improve the upper bound from Theorem \ref{zgornjameja}. 

\begin{theorem} \label{pot-poljuben}
Let $H$ a graph and $n$ a positive integer. If $k=|H|-\alpha(H)$, then 
$$\chi_{\rho}(P_n\circ H) \leq n|H|-\left \lceil \frac{n}{2} \right \rceil \alpha{(H)}-\sum_{i=2}^{k+1}\left \lceil \frac{n}{i+1} \right \rceil-
\sum_{j=k+2}^{|H|+1} \left ( \left \lfloor \frac{\left \lfloor \frac{n}{2} \right \rfloor -1}{\left \lfloor \frac{j}{2} \right \rfloor +1} \right \rfloor + 1 \right ) + |H|+1.$$
\end{theorem}

\begin{proof}
Let $P_n=v_1\ldots v_n$ and $A_H$ be an $\alpha(H)$-set. Clearly, $A_{P_n}=\{v_{2i-1}:i\in \left[\left\lceil\frac{n}{2}\right\rceil\right]\}$ is an $\alpha(P_n)$-set and $A=A_{P_n}\times A_H$ is an $\alpha(P_n\circ H)$-set. Firstly, we color vertices with $k+1$ colors as in the proof of Theorem \ref{zgornjameja}. For this we use 
$$\ell=n|H|-\left \lceil \frac{n}{2} \right \rceil\alpha{(H)}-\sum_{i=2}^{k+1}\left \lceil \frac{n}{i+1} \right \rceil+k+1$$
colors because $\rho_i(P_n)=\left \lceil \frac{n}{i+1} \right \rceil$. 

\noindent In each $G^h$-layer, $h\in A_H$, there exist $\left \lfloor \frac{n}{2} \right \rfloor$ still not colored vertices with an even distance between any two of them. We denote them by $B^h=(V(P_n)-A_{P_n})\times\{h\}$. Additionally we will color with color $j$, $k+2 \leq j \leq |H|+1$, some vertices of exactly one $G^h$-layer, $h\in A_H$. Denote by $G_j^h$ the $G^h$-layer, $h\in A_H$, containing vertices of color $j$, $k+2 \leq j \leq |H|+1$. The biggest distance between two vertices from $B^h$ equals $2\left \lfloor \frac{n}{2} \right \rfloor -2$. Notice that two vertices of $G_j^h$ colored with $j$ must be at least $p_j=2 \left \lfloor \frac{j}{2} \right \rfloor +2$ apart because every second vertex in $G_j^h$-layer, $h\in A_H$, $k+2 \leq j \leq |H|+1$, is already colored (with color 1). Therefore, we can color with $j$ vertices from set
$$\{(v_{2+sp_j},h):0\leq s\leq \left \lfloor \frac{\left \lfloor \frac{n}{2} \right \rfloor -1}{\left \lfloor \frac{j}{2} \right \rfloor +1} \right \rfloor\}.$$ 
Meaning that $t_j= \left ( \left \lfloor \frac{\left \lfloor \frac{n}{2} \right \rfloor -1}{\left \lfloor \frac{j}{2} \right \rfloor +1} \right \rfloor + 1 \right )$ vertices can be colored with color $j$, $k+2 \leq j \leq |H|+1$ in $G_j^h$-layer, $h\in A_H$. \\

\noindent By Theorem~\ref{zgornjameja} we use at most $\ell$ colors for coloring $P_n\circ H$. In addition $t_j$ vertices of $G_j^h$ are colored with $j$, $k+2 \leq j \leq |H|+1$. Meaning that $$\chi_{\rho}(P_n\circ H)\leq \ell-\sum_{j=k+2}^{|H|+1}t_j+|H|-k$$ which completes the proof.
\end{proof}

For $H\cong K_m$ we have $\alpha (H)=1$ and $k=m-1$. The second sum of Theorem \ref{pot-poljuben} has only one term and that is in the case of $j=|H|+1$ so we immediately obtain the following.

\begin{corollary} \label{pot-poln}
For positive integers $n$ and $m$ we have 
$$\chi_{\rho}(P_n\circ K_m) \leq nm-\left \lceil \frac{n}{2} \right \rceil-\sum_{i=2}^{m}\left \lceil \frac{n}{i+1} \right \rceil-
\left \lfloor \frac{\left \lfloor \frac{n}{2} \right \rfloor -1}{\left \lfloor \frac{m+1}{2} \right \rfloor +1} \right \rfloor + m.$$
\end{corollary}

The upper bound from Theorem~\ref{pot-poljuben} is not the best possible in the general case which we can see in the example of coloring $P_8\circ P_6$. Using the coloring described in the proof of that theorem we use 32 colors to color $P_8\circ P_6$, see left part of Figure~\ref{fig:pot8pot6}. But the same graph can be colored with 31 colors, so $\chi_{\rho}(P_8\circ P_6) \leq 31$, see right part of Figure~\ref{fig:pot8pot6}. 

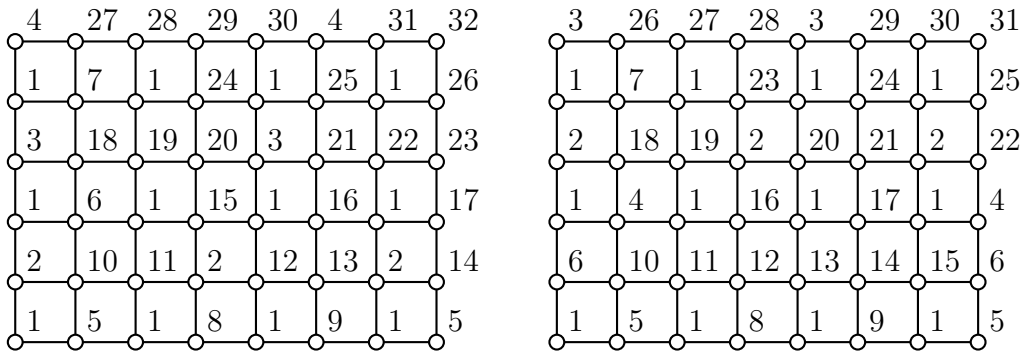
\begin{figure}[H]
\begin{center}
\begin{tikzpicture}[xscale=.8, yscale=.8, style=thick,x=1cm,y=1cm]
\def\vr{3.5pt} 


\path (0,0) coordinate (x1);
\path (1,0) coordinate (x2);
\path (2,0) coordinate (x3);
\path (3,0) coordinate (x4);
\path (4,0) coordinate (x5);
\path (5,0) coordinate (x6);
\path (6,0) coordinate (x7);
\path (7,0) coordinate (x8);

\path (0,1) coordinate (y1);
\path (1,1) coordinate (y2);
\path (2,1) coordinate (y3);
\path (3,1) coordinate (y4);
\path (4,1) coordinate (y5);
\path (5,1) coordinate (y6);
\path (6,1) coordinate (y7);
\path (7,1) coordinate (y8);

\path (0,2) coordinate (z1);
\path (1,2) coordinate (z2);
\path (2,2) coordinate (z3);
\path (3,2) coordinate (z4);
\path (4,2) coordinate (z5);
\path (5,2) coordinate (z6);
\path (6,2) coordinate (z7);
\path (7,2) coordinate (z8);

\path (0,3) coordinate (w1);
\path (1,3) coordinate (w2);
\path (2,3) coordinate (w3);
\path (3,3) coordinate (w4);
\path (4,3) coordinate (w5);
\path (5,3) coordinate (w6);
\path (6,3) coordinate (w7);
\path (7,3) coordinate (w8);

\path (0,4) coordinate (u1);
\path (1,4) coordinate (u2);
\path (2,4) coordinate (u3);
\path (3,4) coordinate (u4);
\path (4,4) coordinate (u5);
\path (5,4) coordinate (u6);
\path (6,4) coordinate (u7);
\path (7,4) coordinate (u8);

\path (0,5) coordinate (v1);
\path (1,5) coordinate (v2);
\path (2,5) coordinate (v3);
\path (3,5) coordinate (v4);
\path (4,5) coordinate (v5);
\path (5,5) coordinate (v6);
\path (6,5) coordinate (v7);
\path (7,5) coordinate (v8);


\draw (x1) -- (x2);
\draw (x2) -- (x3);
\draw (x3) -- (x4);
\draw (x4) -- (x5);
\draw (x5) -- (x6);
\draw (x6) -- (x7);
\draw (x7) -- (x8);

\draw (y1) -- (y2);
\draw (y2) -- (y3);
\draw (y3) -- (y4);
\draw (y4) -- (y5);
\draw (y5) -- (y6);
\draw (y6) -- (y7);
\draw (y7) -- (y8);

\draw (z1) -- (z2);
\draw (z2) -- (z3);
\draw (z3) -- (z4);
\draw (z4) -- (z5);
\draw (z5) -- (z6);
\draw (z6) -- (z7);
\draw (z7) -- (z8);

\draw (w1) -- (w2);
\draw (w2) -- (w3);
\draw (w3) -- (w4);
\draw (w4) -- (w5);
\draw (w5) -- (w6);
\draw (w6) -- (w7);
\draw (w7) -- (w8);

\draw (u1) -- (u2);
\draw (u2) -- (u3);
\draw (u3) -- (u4);
\draw (u4) -- (u5);
\draw (u5) -- (u6);
\draw (u6) -- (u7);
\draw (u7) -- (u8);

\draw (v1) -- (v2);
\draw (v2) -- (v3);
\draw (v3) -- (v4);
\draw (v4) -- (v5);
\draw (v5) -- (v6);
\draw (v6) -- (v7);
\draw (v7) -- (v8);

\draw (x1) -- (y1);
\draw (y1) -- (z1);
\draw (z1) -- (w1);
\draw (w1) -- (u1);
\draw (u1) -- (v1);

\draw (x8) -- (y8);
\draw (y8) -- (z8);
\draw (z8) -- (w8);
\draw (w8) -- (u8);
\draw (u8) -- (v8);

\draw (x2) -- (y2);
\draw (y2) -- (z2);
\draw (z2) -- (w2);
\draw (w2) -- (u2);
\draw (u2) -- (v2);

\draw (x3) -- (y3);
\draw (y3) -- (z3);
\draw (z3) -- (w3);
\draw (w3) -- (u3);
\draw (u3) -- (v3);

\draw (x4) -- (y4);
\draw (y4) -- (z4);
\draw (z4) -- (w4);
\draw (w4) -- (u4);
\draw (u4) -- (v4);

\draw (x5) -- (y5);
\draw (y5) -- (z5);
\draw (z5) -- (w5);
\draw (w5) -- (u5);
\draw (u5) -- (v5);

\draw (x6) -- (y6);
\draw (y6) -- (z6);
\draw (z6) -- (w6);
\draw (w6) -- (u6);
\draw (u6) -- (v6);

\draw (x7) -- (y7);
\draw (y7) -- (z7);
\draw (z7) -- (w7);
\draw (w7) -- (u7);
\draw (u7) -- (v7);

\draw (x1) [fill=white] circle (\vr);
\draw (x2) [fill=white] circle (\vr);
\draw (x3) [fill=white] circle (\vr);
\draw (x4) [fill=white] circle (\vr);
\draw (x5) [fill=white] circle (\vr);
\draw (x6) [fill=white] circle (\vr);
\draw (x7) [fill=white] circle (\vr);
\draw (x8) [fill=white] circle (\vr);

\draw (y1) [fill=white] circle (\vr);
\draw (y2) [fill=white] circle (\vr);
\draw (y3) [fill=white] circle (\vr);
\draw (y4) [fill=white] circle (\vr);
\draw (y5) [fill=white] circle (\vr);
\draw (y6) [fill=white] circle (\vr);
\draw (y7) [fill=white] circle (\vr);
\draw (y8) [fill=white] circle (\vr);

\draw (z1) [fill=white] circle (\vr);
\draw (z2) [fill=white] circle (\vr);
\draw (z3) [fill=white] circle (\vr);
\draw (z4) [fill=white] circle (\vr);
\draw (z5) [fill=white] circle (\vr);
\draw (z6) [fill=white] circle (\vr);
\draw (z7) [fill=white] circle (\vr);
\draw (z8) [fill=white] circle (\vr);

\draw (w1) [fill=white] circle (\vr);
\draw (w2) [fill=white] circle (\vr);
\draw (w3) [fill=white] circle (\vr);
\draw (w4) [fill=white] circle (\vr);
\draw (w5) [fill=white] circle (\vr);
\draw (w6) [fill=white] circle (\vr);
\draw (w7) [fill=white] circle (\vr);
\draw (w8) [fill=white] circle (\vr);

\draw (u1) [fill=white] circle (\vr);
\draw (u2) [fill=white] circle (\vr);
\draw (u3) [fill=white] circle (\vr);
\draw (u4) [fill=white] circle (\vr);
\draw (u5) [fill=white] circle (\vr);
\draw (u6) [fill=white] circle (\vr);
\draw (u7) [fill=white] circle (\vr);
\draw (u8) [fill=white] circle (\vr);

\draw (v1) [fill=white] circle (\vr);
\draw (v2) [fill=white] circle (\vr);
\draw (v3) [fill=white] circle (\vr);
\draw (v4) [fill=white] circle (\vr);
\draw (v5) [fill=white] circle (\vr);
\draw (v6) [fill=white] circle (\vr);
\draw (v7) [fill=white] circle (\vr);
\draw (v8) [fill=white] circle (\vr);

\draw[anchor = south west] (x1) node {$1$};
\draw[anchor = south west] (x3) node {$1$};
\draw[anchor = south west] (x5) node {$1$};
\draw[anchor = south west] (x7) node {$1$};

\draw[anchor = south west] (z1) node {$1$};
\draw[anchor = south west] (z3) node {$1$};
\draw[anchor = south west] (z5) node {$1$};
\draw[anchor = south west] (z7) node {$1$};

\draw[anchor = south west] (u1) node {$1$};
\draw[anchor = south west] (u3) node {$1$};
\draw[anchor = south west] (u5) node {$1$};
\draw[anchor = south west] (u7) node {$1$};

\draw[anchor = south west] (y1) node {$2$};
\draw[anchor = south west] (y4) node {$2$};
\draw[anchor = south west] (y7) node {$2$};

\draw[anchor = south west] (w1) node {$3$};
\draw[anchor = south west] (w5) node {$3$};

\draw[anchor = south west] (v1) node {$4$};
\draw[anchor = south west] (v6) node {$4$};

\draw[anchor = south west] (x2) node {$5$};
\draw[anchor = south west] (x8) node {$5$};

\draw[anchor = south west] (z2) node {$6$};

\draw[anchor = south west] (u2) node {$7$};

\draw[anchor = south west] (x4) node {$8$};
\draw[anchor = south west] (x6) node {$9$};

\draw[anchor = south west] (y2) node {$10$};
\draw[anchor = south west] (y3) node {$11$};
\draw[anchor = south west] (y5) node {$12$};
\draw[anchor = south west] (y6) node {$13$};
\draw[anchor = south west] (y8) node {$14$};

\draw[anchor = south west] (z4) node {$15$};
\draw[anchor = south west] (z6) node {$16$};
\draw[anchor = south west] (z8) node {$17$};

\draw[anchor = south west] (w2) node {$18$};
\draw[anchor = south west] (w3) node {$19$};
\draw[anchor = south west] (w4) node {$20$};
\draw[anchor = south west] (w6) node {$21$};
\draw[anchor = south west] (w7) node {$22$};
\draw[anchor = south west] (w8) node {$23$};

\draw[anchor = south west] (u4) node {$24$};
\draw[anchor = south west] (u6) node {$25$};
\draw[anchor = south west] (u8) node {$26$};

\draw[anchor = south west] (v2) node {$27$};
\draw[anchor = south west] (v3) node {$28$};
\draw[anchor = south west] (v4) node {$29$};
\draw[anchor = south west] (v5) node {$30$};
\draw[anchor = south west] (v7) node {$31$};
\draw[anchor = south west] (v8) node {$32$};


\path (9,0) coordinate (x1);
\path (10,0) coordinate (x2);
\path (11,0) coordinate (x3);
\path (12,0) coordinate (x4);
\path (13,0) coordinate (x5);
\path (14,0) coordinate (x6);
\path (15,0) coordinate (x7);
\path (16,0) coordinate (x8);

\path (9,1) coordinate (y1);
\path (10,1) coordinate (y2);
\path (11,1) coordinate (y3);
\path (12,1) coordinate (y4);
\path (13,1) coordinate (y5);
\path (14,1) coordinate (y6);
\path (15,1) coordinate (y7);
\path (16,1) coordinate (y8);

\path (9,2) coordinate (z1);
\path (10,2) coordinate (z2);
\path (11,2) coordinate (z3);
\path (12,2) coordinate (z4);
\path (13,2) coordinate (z5);
\path (14,2) coordinate (z6);
\path (15,2) coordinate (z7);
\path (16,2) coordinate (z8);

\path (9,3) coordinate (w1);
\path (10,3) coordinate (w2);
\path (11,3) coordinate (w3);
\path (12,3) coordinate (w4);
\path (13,3) coordinate (w5);
\path (14,3) coordinate (w6);
\path (15,3) coordinate (w7);
\path (16,3) coordinate (w8);

\path (9,4) coordinate (u1);
\path (10,4) coordinate (u2);
\path (11,4) coordinate (u3);
\path (12,4) coordinate (u4);
\path (13,4) coordinate (u5);
\path (14,4) coordinate (u6);
\path (15,4) coordinate (u7);
\path (16,4) coordinate (u8);

\path (9,5) coordinate (v1);
\path (10,5) coordinate (v2);
\path (11,5) coordinate (v3);
\path (12,5) coordinate (v4);
\path (13,5) coordinate (v5);
\path (14,5) coordinate (v6);
\path (15,5) coordinate (v7);
\path (16,5) coordinate (v8);


\draw (x1) -- (x2);
\draw (x2) -- (x3);
\draw (x3) -- (x4);
\draw (x4) -- (x5);
\draw (x5) -- (x6);
\draw (x6) -- (x7);
\draw (x7) -- (x8);

\draw (y1) -- (y2);
\draw (y2) -- (y3);
\draw (y3) -- (y4);
\draw (y4) -- (y5);
\draw (y5) -- (y6);
\draw (y6) -- (y7);
\draw (y7) -- (y8);

\draw (z1) -- (z2);
\draw (z2) -- (z3);
\draw (z3) -- (z4);
\draw (z4) -- (z5);
\draw (z5) -- (z6);
\draw (z6) -- (z7);
\draw (z7) -- (z8);

\draw (w1) -- (w2);
\draw (w2) -- (w3);
\draw (w3) -- (w4);
\draw (w4) -- (w5);
\draw (w5) -- (w6);
\draw (w6) -- (w7);
\draw (w7) -- (w8);

\draw (u1) -- (u2);
\draw (u2) -- (u3);
\draw (u3) -- (u4);
\draw (u4) -- (u5);
\draw (u5) -- (u6);
\draw (u6) -- (u7);
\draw (u7) -- (u8);

\draw (v1) -- (v2);
\draw (v2) -- (v3);
\draw (v3) -- (v4);
\draw (v4) -- (v5);
\draw (v5) -- (v6);
\draw (v6) -- (v7);
\draw (v7) -- (v8);

\draw (x1) -- (y1);
\draw (y1) -- (z1);
\draw (z1) -- (w1);
\draw (w1) -- (u1);
\draw (u1) -- (v1);

\draw (x8) -- (y8);
\draw (y8) -- (z8);
\draw (z8) -- (w8);
\draw (w8) -- (u8);
\draw (u8) -- (v8);

\draw (x2) -- (y2);
\draw (y2) -- (z2);
\draw (z2) -- (w2);
\draw (w2) -- (u2);
\draw (u2) -- (v2);

\draw (x3) -- (y3);
\draw (y3) -- (z3);
\draw (z3) -- (w3);
\draw (w3) -- (u3);
\draw (u3) -- (v3);

\draw (x4) -- (y4);
\draw (y4) -- (z4);
\draw (z4) -- (w4);
\draw (w4) -- (u4);
\draw (u4) -- (v4);

\draw (x5) -- (y5);
\draw (y5) -- (z5);
\draw (z5) -- (w5);
\draw (w5) -- (u5);
\draw (u5) -- (v5);

\draw (x6) -- (y6);
\draw (y6) -- (z6);
\draw (z6) -- (w6);
\draw (w6) -- (u6);
\draw (u6) -- (v6);

\draw (x7) -- (y7);
\draw (y7) -- (z7);
\draw (z7) -- (w7);
\draw (w7) -- (u7);
\draw (u7) -- (v7);

\draw (x1) [fill=white] circle (\vr);
\draw (x2) [fill=white] circle (\vr);
\draw (x3) [fill=white] circle (\vr);
\draw (x4) [fill=white] circle (\vr);
\draw (x5) [fill=white] circle (\vr);
\draw (x6) [fill=white] circle (\vr);
\draw (x7) [fill=white] circle (\vr);
\draw (x8) [fill=white] circle (\vr);

\draw (y1) [fill=white] circle (\vr);
\draw (y2) [fill=white] circle (\vr);
\draw (y3) [fill=white] circle (\vr);
\draw (y4) [fill=white] circle (\vr);
\draw (y5) [fill=white] circle (\vr);
\draw (y6) [fill=white] circle (\vr);
\draw (y7) [fill=white] circle (\vr);
\draw (y8) [fill=white] circle (\vr);

\draw (z1) [fill=white] circle (\vr);
\draw (z2) [fill=white] circle (\vr);
\draw (z3) [fill=white] circle (\vr);
\draw (z4) [fill=white] circle (\vr);
\draw (z5) [fill=white] circle (\vr);
\draw (z6) [fill=white] circle (\vr);
\draw (z7) [fill=white] circle (\vr);
\draw (z8) [fill=white] circle (\vr);

\draw (w1) [fill=white] circle (\vr);
\draw (w2) [fill=white] circle (\vr);
\draw (w3) [fill=white] circle (\vr);
\draw (w4) [fill=white] circle (\vr);
\draw (w5) [fill=white] circle (\vr);
\draw (w6) [fill=white] circle (\vr);
\draw (w7) [fill=white] circle (\vr);
\draw (w8) [fill=white] circle (\vr);

\draw (u1) [fill=white] circle (\vr);
\draw (u2) [fill=white] circle (\vr);
\draw (u3) [fill=white] circle (\vr);
\draw (u4) [fill=white] circle (\vr);
\draw (u5) [fill=white] circle (\vr);
\draw (u6) [fill=white] circle (\vr);
\draw (u7) [fill=white] circle (\vr);
\draw (u8) [fill=white] circle (\vr);

\draw (v1) [fill=white] circle (\vr);
\draw (v2) [fill=white] circle (\vr);
\draw (v3) [fill=white] circle (\vr);
\draw (v4) [fill=white] circle (\vr);
\draw (v5) [fill=white] circle (\vr);
\draw (v6) [fill=white] circle (\vr);
\draw (v7) [fill=white] circle (\vr);
\draw (v8) [fill=white] circle (\vr);

\draw[anchor = south west] (x1) node {$1$};
\draw[anchor = south west] (x3) node {$1$};
\draw[anchor = south west] (x5) node {$1$};
\draw[anchor = south west] (x7) node {$1$};

\draw[anchor = south west] (z1) node {$1$};
\draw[anchor = south west] (z3) node {$1$};
\draw[anchor = south west] (z5) node {$1$};
\draw[anchor = south west] (z7) node {$1$};

\draw[anchor = south west] (u1) node {$1$};
\draw[anchor = south west] (u3) node {$1$};
\draw[anchor = south west] (u5) node {$1$};
\draw[anchor = south west] (u7) node {$1$};

\draw[anchor = south west] (x2) node {$5$};
\draw[anchor = south west] (x8) node {$5$};

\draw[anchor = south west] (y1) node {$6$};
\draw[anchor = south west] (y8) node {$6$};

\draw[anchor = south west] (z2) node {$4$};
\draw[anchor = south west] (z8) node {$4$};

\draw[anchor = south west] (w1) node {$2$};
\draw[anchor = south west] (w4) node {$2$};
\draw[anchor = south west] (w7) node {$2$};

\draw[anchor = south west] (v1) node {$3$};
\draw[anchor = south west] (v5) node {$3$};

\draw[anchor = south west] (u2) node {$7$};

\draw[anchor = south west] (x4) node {$8$};
\draw[anchor = south west] (x6) node {$9$};
\draw[anchor = south west] (y2) node {$10$};
\draw[anchor = south west] (y3) node {$11$};
\draw[anchor = south west] (y4) node {$12$};
\draw[anchor = south west] (y5) node {$13$};
\draw[anchor = south west] (y6) node {$14$};
\draw[anchor = south west] (y7) node {$15$};
\draw[anchor = south west] (z4) node {$16$};
\draw[anchor = south west] (z6) node {$17$};
\draw[anchor = south west] (w2) node {$18$};
\draw[anchor = south west] (w3) node {$19$};
\draw[anchor = south west] (w5) node {$20$};
\draw[anchor = south west] (w6) node {$21$};
\draw[anchor = south west] (w8) node {$22$};
\draw[anchor = south west] (u4) node {$23$};
\draw[anchor = south west] (u6) node {$24$};
\draw[anchor = south west] (u8) node {$25$};
\draw[anchor = south west] (v2) node {$26$};
\draw[anchor = south west] (v3) node {$27$};
\draw[anchor = south west] (v4) node {$28$};
\draw[anchor = south west] (v6) node {$29$};
\draw[anchor = south west] (v7) node {$30$};
\draw[anchor = south west] (v8) node {$31$};

\end{tikzpicture}
\end{center}
\caption{Packing coloring for $P_8 \circ P_6$ using 32 colors according to Theorem \ref{pot-poljuben} (a) and 31 colors (b) (not all edges of a graph are drawn).}
\label{fig:pot8pot6}
\end{figure}

Another example can be constructed as follows. Let $n_t=1+lcm(2,3,\ldots,t+1)$, $H\ncong\overline{K}_m$ a graph and $k=|H|-\alpha(H)$. Notice that $n_t$ is chosen in such a way that every $\rho_i(P_{n_t})$-set, $1\leq i\leq t$, contains the first and the last vertex of $P_{n_t}$. If $t-1>k$, then we cannot obtain $\rho_i(P_{n_t})$ vertices of color $i$ in $P_{n_t}\circ H$ for some $2\leq i\leq t$ and the upper bound of Theorem \ref{pot-poljuben} is not exact.

\end{document}